\documentclass[12pt]{amsart}

\usepackage{enumerate}
\usepackage{a4wide}
\usepackage{amsmath, amssymb,amsthm}

\newtheorem{theorem}{Theorem}[section]
\newtheorem{lemma}[theorem]{Lemma}

\newtheorem{proposition}[theorem]{Proposition}

\newtheorem*{maintheorem}{Main Theorem}

\newcommand{\C}{\mathbb{C}}
\newcommand{\R}{\mathbb{R}}
\newcommand{\CH}{\mathbb{C}\mathrm{H}}
\newcommand{\CP}{\mathbb{C}\mathrm{P}}

\DeclareMathOperator{\Id}{Id}
\DeclareMathOperator{\Ad}{Ad}
\DeclareMathOperator{\ad}{ad}
\DeclareMathOperator{\Image}{Im}

\DeclareMathOperator{\Exp}{Exp}

\newcommand{\g}[1]{\ensuremath{\mathfrak{#1}}}
\newcommand{\s}[1]{\ensuremath{\mathsf{#1}}}
\newcommand{\cal}[1]{\ensuremath{\mathcal{#1}}}

\begin{document}

\title[Homogeneous Lagrangian foliations]{Homogeneous Lagrangian foliations\\on complex space forms}

\author[J.~C. D\'\i{}az-Ramos]{Jos\'{e} Carlos D\'\i{}az-Ramos}
\address{Department of Mathematics,
University of Santiago de Compostela, Spain.}
\email{josecarlos.diaz@usc.es}

\author[M. Dom\'{\i}nguez-V\'{a}zquez]{Miguel Dom\'{\i}nguez-V\'{a}zquez}
\address{Department of Mathematics,
University of Santiago de Compostela, Spain.}
\email{miguel.dominguez@usc.es}

\author[T. Hashinaga]{Takahiro Hashinaga}
\address{Faculty of Education, Saga University, Saga, Japan.}
\email{hashinag@cc.saga-u.ac.jp}

\subjclass[2010]{Primary 53D12; Secondary 53C12, 53C35, 57S20}
\keywords{Homogeneous Lagrangian foliation, Lagrangian submanifold, isometric action, complex space form, complex hyperbolic space, horocycle foliation}

\begin{abstract}
We classify holomorphic isometric actions on complex space forms all whose orbits are Lagrangian submanifolds, up to orbit equivalence. The only examples are Lagrangian affine subspace foliations of complex Euclidean spaces, and Lagrangian horocycle foliations of complex hyperbolic spaces.
\end{abstract}

\thanks{The first and second authors have been supported by the projects PID2019-105138GB-C21 (AEI/FEDER, Spain) and ED431C 2019/10, ED431F 2020/04 (Xunta de Galicia, Spain).  The second author acknowledges support of the Ram\'on y Cajal program of the Spanish State Research Agency. The third author has been supported by JSPS KAKENHI Grant Number 16K17603. This work was partly supported by Osaka City University Advanced Mathematical Institute (MEXT Joint Usage/Research Center on Mathematics and Theoretical Physics JPMXP0619217849).}

\maketitle
\vspace{-2ex}

\section{Introduction}
\enlargethispage{2ex}\thispagestyle{empty}

Let $({M},g,J)$ be a K\"{a}hler manifold with Riemannian metric $g$ and complex structure~$J$.
As a K\"{a}hler manifold, $M$ carries a natural symplectic structure defined as $\omega(X,Y)=g(X,JY)$, which turns it into a symplectic manifold.
A submanifold $N$ of ${M}$ is said to be \emph{Lagrangian} if the symplectic form $\omega$ vanishes on $TN$ and the dimension of $N$ is half the dimension of $M$.
In the case of a K\"{a}hler manifold this is equivalent to saying that $N$ is totally real and of half the (real) dimension of $M$, or equivalently, $J T_p N= \nu_p M$ for all $p\in N$, and where $\nu N$ denotes the normal bundle of $N$.

A foliation on a symplectic manifold is called Lagrangian if all its leaves are Lagrangian submanifolds. These are important objects in symplectic geometry and topology, as exemplified in Weinstein's foundational results~\cite{Weinstein} or in more recent contributions~\cite{Gavela}. Now, if $M$ is a K\"ahler manifold, a foliation on $M$ is said to be \emph{homogeneous} if it is the family of orbits of an action of a connected Lie group of automorphisms of the K\"ahler structure of $M$, that is, of holomorphic isometries of $M$. This article is motivated by the problem of investigating homogeneous Lagrangian foliations on K\"ahler manifolds. In this sense, our aim here is to provide the classification of homogeneous Lagrangian foliations on complex space forms, that is, complete simply connected K\"ahler manifolds with constant holomorphic sectional curvature: complex Euclidean spaces $\C^n$, complex projective spaces $\CP^n$, and complex hyperbolic spaces~$\CH^n$.

\begin{maintheorem}
A homogeneous Lagrangian foliation on a complex space form is holomorphically congruent to a Lagrangian affine subspace foliation on $\C^n$ or to a Lagrangian horocycle foliation on $\CH^n$.
\end{maintheorem}

Here, by Lagrangian affine subspace foliation we mean the orbit foliation of the action of a subgroup of translations of $\C^n$ given by a Lagrangian subspace $V\cong\R^n$ of~$\C^n$. More interesting is the example in the complex hyperbolic space, which is constructed as follows.
As a symmetric space, the complex hyperbolic space can be written as $\CH^n\cong\s{SU}(1,n)/\s{U}(n)$, where $K\cong\s{U}(n)$ is a maximal compact subgroup of $G=\s{SU}(1,n)$.
Let $\g{g}=\g{k}\oplus\g{p}$ be the Cartan decomposition of the Lie algebra $\g{g}=\g{su}(1,n)$ with respect to $\g{k}\cong\g{u}(n)$, where $\g{p}$ is the orthogonal complement of $\g{k}$ in $\g{g}$ with respect to the Killing form of $\g{g}$.
If we choose $\g{a}$ a maximal abelian subspace of $\g{p}$, the restricted root space decomposition of $\g{g}$ can be written as
$\g{g}=\g{g}_{-2\alpha}\oplus\g{g}_{-\alpha}\oplus\g{g}_0
\oplus\g{g}_\alpha\oplus\g{g}_{2\alpha}$.
The fact that $\CH^n$ is a K\"{a}hler manifold implies that $\g{g}_\alpha$ is a complex vector space of dimension $n-1$.
We refer to Subsection~\ref{subsec:chn} for further details.
We define $\g{h}=\g{l}\oplus\g{g}_{2\alpha}$, where $\g{l}$ is a Lagrangian subspace of $\g{g}_\alpha\cong\C^{n-1}$. We denote by $H$ the connected Lie subgroup of $\s{SU}(1,n)$ whose Lie algebra is $\g{h}$.
We call the foliation induced by $H$ on $\CH^n$ the \emph{Lagrangian horocycle foliation}.
It can be proved that all the orbits of $H$ are congruent to each other, and that they all are Lagrangian submanifolds of $\CH^n$. The Lagrangian horocycle foliation is also an example of a homogeneous polar foliation~\cite{DDK:mathz}.

Our result fits into the framework of the investigation of homogeneous Lagrangian submanifolds in K\"ahler manifolds. In this setting, Bedulli and Gori~\cite{BG:cag} studied general properties of these objects, and classified compact homogeneous Lagrangian submanifolds in $\CP^n$ arising as orbits of simple compact subgroups of $\mathsf{SU}(n+1)$. The case of non-simple subgroups constitutes still an open problem (see also~\cite{PP:tohoku}). Interestingly, there is a complete classification of compact homogeneous Lagrangian submanifolds in complex hyperquadrics, due to Ma and Ohnita~\cite{MO:mathz}. Also, Kajigaya and the third author~\cite{HK} derived a classification of homogeneous Lagrangian submanifolds of complex hyperbolic spaces arising as orbits of subgroups of the solvable Iwasawa group of $\CH^n$, and established a correspondence between compact homogeneous Lagrangian submanifolds in $\CH^n$, in $\C^n$, and in $\CP^{n-1}$. There are also some partial classifications of homogeneous totally real submanifolds in Hermitian symmetric spaces under additional geometric assumptions~(cf.~\cite{BCO}).

The starting point for the proof of the Main Theorem is a known result (see Proposition~\ref{prop:abelian}) that ensures that a homogeneous Lagrangian foliation on a K\"ahler manifold must be induced by an abelian Lie group. This easily implies the non-existence of such foliations on complex projective spaces (see~\S\ref{subsec:cpn}). The non-compact cases, and especially the hyperbolic one, are much more involved. In both settings, the gist of the proof is to deal with the possible existence of helicoidal isometries in the group. In the hyperbolic case (see~\S\ref{subsec:chn_proof}), a first observation is that the group acting is contained in a maximally non-compact Borel subgroup of $\s{SU}(1,n)$. However, dealing with isometries with toroidal component is still rather technical. In our proof, the introduction of a certain suitable self-adjoint operator (see~\eqref{eq:Phi}) turned out to be crucial. We believe that generalizing this idea would be fundamental in order to address the problem in Hermitian symmetric spaces of non-compact type and higher rank. The classification of homogeneous Lagrangian foliations on such spaces seems to be an interesting problem in view of the scarcity of examples in the rank one case and as a first approximation to the more general study of homogeneous Lagrangian submanifolds. But it is also a challenging problem, inasmuch as homogeneous Lagrangian foliations induced by groups of isometries without rotational components (i.e.\ contained in the corresponding solvable Iwasawa group) have not been classified yet. 

This paper is organized as follows. Section~\ref{sec:prelim} contains some preliminaries regarding the complex hyperbolic space and isometric actions with isotropic orbits on K\"ahler manifolds. The proof of the Main Theorem is the content of Section~\ref{sec:proof}.
\smallskip

\textbf{Acknowledgments.} The authors would like to thank an anonymous referee for helpful suggestions on an earlier version of this article.

\section{Preliminaries}\label{sec:prelim}
In this section, we collect some basic facts regarding the complex hyperbolic space $\CH^n$ and the algebraic structure of its isometry group (in \S\ref{subsec:chn}), and we include a proof of a known result about actions with isotropic orbits in the K\"ahler setting (in \S\ref{subsec:isotropic}).

\subsection{The complex hyperbolic space}\label{subsec:chn}\hfill

We denote by $\CH^n$ the complex hyperbolic space with constant negative holomorphic sectional curvature.
It is well known that it can be realized as a Hermitian symmetric space $G/K$, where $G=\s{SU}(1,n)$ is, up to a finite quotient, the connected component of the identity of the isometry group of $\CH^n$, and $K=\s{S}(\s{U}(1) \times \s{U}(n))$ is the isotropy group of $G$ at some point $o \in \CH^n$. Moreover, $G=\s{SU}(1,n)$ coincides (again, up to finite quotient) with the group of holomorphic isometries of $\CH^n$.

Let us denote by $\g{g}=\g{su}(1,n)$ and $\g{k}=\g{s}(\g{u}(1)\oplus\g{u}(n))\cong\g{u}(n)$ the Lie algebras of $G$ and $K$, respectively, and let $\g{g}=\g{k} \oplus \g{p}$ be the corresponding Cartan decomposition, where $\g{p}$ is the orthogonal complement of $\g{k}$ in $\g{g}$ with respect to the Killing form $\cal{B}$ of $\g{g}$. Recall that we can identify $\g{p}$ with the tangent space $T_o \CH^n$ as vector spaces. Denote by $\theta$ the corresponding Cartan involution, whose $+1$ (resp.\ $-1$) eigenspace is $\g{k}$ (resp.\ $\g{p}$).
Then $\cal{B}_\theta(X, Y):=-\cal{B}(\theta X, Y)$, $X$, $Y \in \g{g}$, defines a positive definite inner product on $\g{g}$ that satisfies
$\cal{B}_\theta(\ad(X)Y,Z) =-\cal{B}_\theta(Y, \ad(\theta X)Z)$,
for all $X$, $Y$, $Z \in \g{g}$.
The $\Ad(K)$-invariant inner product $\cal{B}_\theta$ on $\g{p}$ induces the Riemannian metric on $\CH^n\cong G/K$.

As for any other irreducible Hermitian symmetric space $G/K$, the Lie subalgebra $\g{k}$ of $\g{g}$ has a $1$-dimensional center $Z(\g{k})$. Then, there is an element $\zeta \in Z(\g{k})$, which is unique up to sign, such that $J_o:=\ad(\zeta)\vert_{\g{p}}$ defines an $\Ad(K)$-invariant (orthogonal) complex structure on $\g{p}$. Since $\Ad(K)$-invariant structures on $\g{p}$ induce $G$-invariant structures on $G/K$, $J_o$ determines a $G$-invariant (orthogonal) complex structure $J$ on $G/K$. Then $G/K$ becomes isomorphic to $\CH^n$ as a K\"ahler manifold.

Let $\g{a}$ be a maximal abelian subspace in $\g{p}$, which is $1$-dimensional, as $\CH^n$ is a rank one symmetric space.
This abelian subspace $\g{a}$ induces a  $\cal{B}_\theta$-orthogonal direct sum decomposition
$\g{g}=\g{g}_{-2 \alpha} \oplus \g{g}_{-\alpha} \oplus \g{g}_{0} \oplus \g{g}_{\alpha} \oplus \g{g}_{2\alpha}$, called a restricted root space decomposition of $\g{g}$, where
$\g{g}_{\lambda}=\{ X \in \g{g}: \mathrm{ad}(H)X=\lambda(H)X,\text{ for all $H \in \g{a}$} \}$ for each $\lambda \in \g{a}^*$.
Let $\g{k}_0=\g{g}_0 \cap \g{k} \cong \g{u} (n-1)$ be the centralizer of $\g{a}$ in $\g{k}$.
It turns out that $\g{g}_{0}=\g{k}_0 \oplus \g{a}$,  $\g{g}_{\alpha}$ is normalized by $\g{k}_0$, and both $\g{a}$ and $\g{g}_{2\alpha}$ are centralized by $\g{k}_0$.
Set $\g{n}=\g{g}_{\alpha} \oplus \g{g}_{2\alpha}$, which is a nilpotent Lie subalgebra of $\g{g}$ isomorphic to the $(2n-1)$-dimensional Heisenberg Lie algebra.
Then the vector space direct sum $\g{g}=\g{k}\oplus\g{a} \oplus \g{n}$ is an Iwasawa decomposition of $\g{g}$, which induces, at the Lie group level, a decomposition $G\cong K\times A\times N$ as a Cartesian product.

It follows from the Iwasawa decomposition at the Lie group level that the connected closed subgroup $AN$ of $G$ with Lie algebra $\g{a} \oplus \g{n}$ acts simply transitively on $G/K\cong\CH^n$ by the natural left action.
In particular, the map $\Phi\colon AN \rightarrow \CH^n$ defined by $\Phi(g)=g(o)$ is a diffeomorphism.
Thus, we can equip $AN$ with the K\"{a}hler structure so that $\Phi$ is a holomorphic isometry.
We will denote by $\langle\cdot , \cdot \rangle$ the induced Riemannian metric on $AN$, which turns out to be left-invariant on $AN$, and by $J$ the induced left-invariant complex structure on $AN$, given by $J=\Phi^{-1}_* \circ J \circ \Phi_*$.
It turns out that the  inner product $\langle\cdot , \cdot \rangle=\Phi^* \cal{B}_\theta$ on $\g{a}\oplus\g{n}\cong T_eAN$ is related to the inner product $\cal{B}_\theta$ on $\g{g}$ by
$\langle X,Y\rangle=\cal{B}_\theta(X_{\g{a}},Y_{\g{a}})+
\frac{1}{2}\cal{B}_\theta (X_{\g{n}},Y_{\g{n}})$,
where subscripts mean the $\g{a}$ and $\g{n}$ components, respectively, for any $X$, $Y\in\g{a}\oplus\g{n}$.
In particular, for any $T\in\g{k}_0$ and $X$, $Y\in\g{n}$ we have
\[
\langle \ad(T)X,Y\rangle = -\langle X,\ad(T)Y\rangle,
\]
that is, the elements of $\ad(\g{k}_0)\vert_{\g{n}}$ are skew-symmetric with respect to $\langle\cdot , \cdot \rangle$ (and $\cal{B}_\theta$).
One can also prove that $\g{g}_{\alpha}$ is $J$-invariant, namely, $\g{g}_{\alpha}$ is a complex subspace in $\g{a} \oplus \g{n}$, and $J \g{a} = \g{g}_{2\alpha}$.
Let $B \in \g{a}$ be a unit vector and define $Z:=JB \in \g{g}_{2\alpha}$.
Then, the Lie bracket of $\g{a} \oplus \g{n}$ is given by
\begin{align*}
[aB+U+xZ,\, bB+V+yZ]=-\frac{b}{2}U+\frac{a}{2}V+\bigl( -bx+ay+\langle JU, V \rangle \bigr)Z,
\end{align*}
where $a, b, x, y \in \R$ and $U$, $V \in \g{g}_{\alpha}$.

\begin{lemma}\label{lemma:adJ}
	We have $\mathrm{ad}(T) J\vert_{\g{g}_\alpha} =J  \mathrm{ad}(T)\vert_{\g{g}_\alpha}$, for any $T\in\g{k}_0$.
\end{lemma}

\begin{proof}
Let $T\in\g{k}_0$ and $U$, $V\in\g{g}_\alpha$. Since $[JU,V]\in\g{g}_{2\alpha}$ and $\g{k}_0$ normalizes $\g{g}_\alpha$ and centralizes $\g{g}_{2\alpha}$, the Jacobi identity gives
\begin{align*}
\langle J[T,JU],V\rangle Z
&{}=[[T,JU],V]
=-[[JU,V],T]-[[V,T],JU]\\
&{}=-\langle J[V,T],JU\rangle Z
=-\langle [T,U],V\rangle Z.
\end{align*}
Hence, $J[T,JU]=-[T,U]$, which is equivalent to our claim.
\end{proof}

\subsection{Actions with isotropic orbits}\label{subsec:isotropic}\hfill

We recall here some known facts about K\"ahler geometry from the symplectic viewpoint, and Lie group actions with isotropic orbits.

If $M$ is a K\"ahler manifold with metric $\langle\cdot,\cdot\rangle$ and complex structure $J$, then $\omega(X,Y):=\langle X,JY\rangle$ defines a symplectic structure on $M$. Holomorphic isometries of $M$ are then symplectomorphisms of $(M,\omega)$. A submanifold $N$ of a symplectic manifold $(M,\omega)$ is called isotropic if the symplectic form  of $M$ vanishes on tangent vectors to $N$, i.e.\ $\omega(v,w)=0$ for all $v$, $w\in T_pN$, $p\in N$. In the context of K\"ahler geometry, isotropic submanifolds are usually called totally real, as in this case the condition translates into $J T_pN\perp T_pN$, for each $p\in N$. An isotropic submanifold $N$ is called Lagrangian if it is of the maximal possible dimension among isotropic submanifolds, i.e.\ $2\dim N=\dim M$, where we consider real dimensions.

We will make use of a well-known result on actions on symplectic manifolds~\cite[Proposition~III.2.12]{A}. Here, we state it in the setting of K\"ahler geometry, and include a proof for completeness. 

\begin{proposition}\label{prop:abelian}
Let $H$ be a connected Lie group acting almost effectively and by holomorphic isometries on a K\"ahler manifold $M$. If all $H$-orbits are isotropic, then $H$ is abelian.
\end{proposition}
\begin{proof}
For each $X\in\g{h}$ in the Lie algebra of $H$, we denote by $X^*$ the associated fundamental Killing vector field on $M$, given by $X^*_p= \frac{d}{dt}\vert_{t=0}\Exp (tX)(p)$, $p\in M$, where $\Exp$ denotes the Lie group exponential map of $H$. Since the $H$-action is by holomorphic isometries, it preserves the K\"ahler form $\omega$, and hence $\cal{L}_{X^*}\omega=0$ for any $X\in\g{h}$. Since $\omega$ is closed, for any $X$, $Y\in\g{h}$ and any smooth vector field $W$ on $M$, we get
\begin{align*}
0=d\omega(X^*,Y^*,W)&=
\begin{aligned}[t]
&X^*\omega(Y^*,W)+Y^*\omega(W,X^*)+W\omega(X^*,Y^*)
\\
&-\omega([X^*,Y^*],W)-\omega([Y^*,W],X^*)-\omega([W,X^*],Y^*)
\end{aligned}
\\
&=W\omega(X^*,Y^*)+\omega([X^*,Y^*],W)=\omega([X^*,Y^*],W),
\end{align*}
where in the third equality we have taken $\cal{L}_{X^*}\omega=\cal{L}_{Y^*}\omega=0$ into account, and in the last one we have used that $\omega(X^*,Y^*)=0$, as all $H$-orbits are isotropic by assumption. Since $\omega$ is non-degenerate, we obtain $[X^*,Y^*]=0$ for any $X$, $Y\in\g{h}$.
But $X\in\g{h}\mapsto X^*\in\Gamma(TM)$ is a Lie algebra anti-homomorphism, which is one-to-one due to the assumption that $H$ acts almost effectively on $M$. Thus, we get $[X,Y]^*=0$, and then, $[X,Y]=0$, for any $X$, $Y\in\g{h}$. This shows that $\g{h}$ is abelian, and hence, also $H$, since it is~connected.
\end{proof}

\section{Proof of the Main theorem}\label{sec:proof}
We divide the proof in three subsections according to the sign of the curvature of the ambient complex space form.

\subsection{Homogeneous Lagrangian foliations on complex projective spaces}\label{subsec:cpn}\hfill

The non-existence of homogeneous Lagrangian foliations on complex projective spaces is a direct consequence of Proposition~\ref{prop:abelian}. Indeed, the Lie group $\s{SU}(n+1)$ is (up to a finite quotient) the connected component of the identity of the isometry group the complex projective space $\CP^n$, which coincides with the group of holomorphic isometries of $\CP^n$. Let $H$ be a connected Lie subgroup of $\s{SU}(n+1)$ acting on $\CP^n$ with Lagrangian orbits. By Proposition~\ref{prop:abelian}, $H$ is abelian. Therefore, $H$ is contained in a maximal torus $T^n$ of the compact Lie group $\s{SU}(n+1)$. But the action of any maximal torus of $\s{SU}(n+1)$ on $\CP^n$ has a fixed point (since by rank reasons $T^n$ is contained in a maximal proper compact subgroup $\s{U}(n)$ of $\s{SU}(n+1)$, which necessarily fixes a point). Hence, also the $H$-action has a fixed point, which contradicts the assumption that all orbits are Lagrangian.

\subsection{Homogeneous Lagrangian foliations on complex Euclidean spaces}\label{subsec:cn}\hfill

We recall first some well-known facts about the isometry group of a Euclidean space $\R^{n}$. The connected component $I^0(\R^n)$ of the identity element of the isometry group of $\R^{n}$ is isomorphic to the semi-direct product $\s{SO}(n)\times_\Phi\R^n$,
where $\Phi\colon \s{SO}(n)\to\mathop{\rm Aut}(\R^n)$ is given by $\Phi(a)({v})=a{v}$.
Hence, the group operation is given by the
formula $(a,{v})(b,{w})=(ab,{v}+a{w})$. The isometry group
of $\R^n$ acts on $\R^n$ in the obvious way by
$(a,{v}){x}=a{x}+{v}$.
The Lie algebra of $I^0(\R^n)$ is the semi-direct sum
$\g{so}(n)\oplus_\phi \R^n$ where
$\phi\colon\g{so}(n)\to\mathop{\rm Der}(\R^n)$ is given by
$\phi(X)({v})=X{v}$.
Thus, the Lie bracket is given by
$[X+{v},Y+{w}]=XY-YX+X{w}-Y{v}$,
and the adjoint representation is
$\Ad(a,{v})(X+{w})=aXa^{-1}-aXa^{-1}{v}+a{w}$.

Now, let us assume that $H$ is a connected Lie subgroup of the group $\s{U}(n)\times_\Phi\C^n$ of holomorphic isometries of $\C^n$ acting on $\C^n$ in such a way that all its orbits are Lagrangian. Of course, $\s{U}(n)\times_\Phi\C^n$ is a Lie subgroup of the group $I^0(\C^n)=\s{SO}(2n)\times_\Phi\C^n$ of orientation-preserving isometries of $\C^n\cong\R^{2n}$.
From Proposition~\ref{prop:abelian}, $H$ is abelian.
We denote by $\pi\colon\g{so}(2n)\oplus_\phi\C^{n}\to\g{so}(2n)$ the projection onto the first component, which is a Lie algebra homomorphism. 
We define $V=\C^n\cap\g{h}$ the pure translational part of $\g{h}$.
Note that $\dim\g{h}=\dim\pi(\g{h})+\dim V$.
Since $\g{h}$ is abelian, so is $\pi(\g{h})\subset\g{so}(2n)$. Our goal in the rest of this subsection will be to show that $\pi(\g{h})=0$, or equivalently, $\g{h}=V$.

We define the map
\[
\xi\colon\pi(\g{h})\to\C^n\ominus V,\quad
X\mapsto \xi(X),
\]
by the requirement $X+\xi(X)\in\g{h}$.
In this subsection, we use the symbol $\ominus$ to denote orthogonal complement.
This map is well-defined because $X+\xi(X)$, $X+\xi'(X)\in\g{h}$ implies
$\xi(X)-\xi'(X)=(X+\xi(X))-(X+\xi'(X))\in\g{h}$, and hence, $\xi(X)=\xi'(X)$.

We now follow the approach given in~\cite[Theorem~2.1]{Di02}.
Since $\pi(\g{h})$ is a commuting family of skew-symmetric endomorphisms, we can define
\[
\g{h}_\lambda=\{{v}\in\C^n:X^2{v}=-\lambda(X)^2{v},\text{ for each $X\in\pi(\g{h})$}\},
\]
where $\lambda\in\pi(\g{h})^*$ is a 1-form.
If $\Sigma\subset\pi(\g{h})^*$ denotes the set of all non-zero 1-forms $\lambda$ for which $\g{h}_\lambda\neq 0$,
then we have the $\pi(\g{h})$-invariant orthogonal decomposition
\[
\C^n=\bigoplus_{\lambda\in\Sigma\cup\{0\}}\g{h}_\lambda,
\quad\text{ where }\quad
\g{h}_0=\bigcap_{X\in\pi(\g{h})}\ker X.
\]
We denote $\g{h}_0^\perp=\oplus_{\lambda\in\Sigma}\g{h}_\lambda$, the orthogonal complement of $\g{h}_0$ in $\C^n$, which is also $\pi(\g{h})$-invariant.
In particular, $\pi(\g{h})(\C^n)=\g{h}_0^\perp$, which implies that $\pi(\g{h})$ is isomorphic to an abelian Lie subalgebra of $\g{so}(\g{h}_0^\perp)$.

An element $X\in\pi(\g{h})$ is called regular if $\lambda(X)\neq 0$ for all $\lambda\in\Sigma$.
The subset of regular elements is open and dense in $\pi(\g{h})$.
If $X\in\pi(\g{h})$ is regular, then $\ker X=\g{h}_0$.
Moreover, since $X$ is skew-symmetric we have the orthogonal decomposition $\C^n=(\ker X)\oplus(\Image X)$.
Thus, $\Image X=\pi(\g{h})(\C^n)=\g{h}_0^\perp$.

Let $X\in\pi(\g{h})$ and ${v}\in V$.
Since $\g{h}$ is abelian, we have $0=[X+\xi(X),{v}]=X{v}$. Hence, $V\subset\g{h}_0$.
If $X$, $Y\in\pi(\g{h})$, commutativity implies $0=[X+\xi(X),\, Y+\xi(Y)]=X\xi(Y)-Y\xi(X)$.

\begin{lemma}\label{lemma:c}
	There exists ${c}\in\C^n$ such that $\xi(X)-X{c}\in\g{h}_0$ for all $X\in\pi(\g{h})$.
\end{lemma}

\begin{proof}
	Let $X\in\pi(\g{h})$ be a regular element.
	Since $\C^n=(\ker X)\oplus(\Image X)$, there exists ${c}\in\C^n$ such that $\xi(X)-X{c}\in\ker X=\g{h}_0$.
	Now, let $Y\in\pi(\g{h})$ be arbitrary.
	Since $\pi(\g{h})$ is abelian,
	\[
	0=X\xi(Y)-Y\xi(X)=X\xi(Y)-YX{c}=X(\xi(Y)-Y{c}),
	\]
	and thus, $\xi(Y)-Y{c}\in\ker X=\g{h}_0$, as we wanted to show.
\end{proof}

\begin{proposition}
	We have $\g{h}=V$, and hence the action of $H$ on $\C^n$ is by translations by vectors in a Lagrangian subspace of $\C^n$.
\end{proposition}
\begin{proof}
If $X\in\pi(\g{h})$, then $\Ad(\Id,{c})(X+\xi(X))=X-X{c}+\xi(X)$.
Moreover $\Ad(\Id,{c})(V)=V\subset\g{h}_0$.
Therefore, if we define  
\[
\g{w}=\xi(\pi(\g{h}))\oplus V,
\]
by Lemma~\ref{lemma:c} and after conjugation by $(\Id,{c})$, we can assume $\g{w}\subset\g{h}_0$.

The tangent space of the orbit of $H$ through the origin is precisely $T_{0}(H\cdot{0})=\g{w}$.
By hypothesis, this orbit is Lagrangian in $\C^n$. This means that $\g{w}\cong\R^n$ is a Lagrangian subspace of $\C^n=\g{w}\oplus i\g{w}$.

Let $v\in V$ and $X\in\pi(\g{h})$ be a regular element.
We consider $q\in\g{h}_0^\perp=\Image X$.
Then, there exists $p\in\C^n$ such that $q=Xp$.
The fact that $T_{p}(H\cdot p)$ and $\g{w}$ are totally real, implies
$\langle iv,q\rangle=\langle iv,Xp\rangle=\langle iv,Xp+\xi(X)\rangle=0$.
Hence, $i V$ is orthogonal to $\g{h}_0^\perp$, and thus, $\C V$, the complexification of $V$, satisfies $\C V\subset \g{h}_0$.
Therefore, $\C V\oplus\xi(\pi(\g{h}))\subset\g{h}_0$, and then,
$\dim\g{h}_0\geq 2\dim V+\dim\xi(\pi(\g{h}))$.

We have
\begin{align*}
2n
&{}=\dim\C^n=\dim\g{h}_0+\dim\g{h}_0^\perp
\geq 2\dim V+\dim\xi(\pi(\g{h}))+\dim\g{h}_0^\perp
\\
&{}=\dim\g{w}+\dim V+\dim\g{h}_0^\perp
=n+\dim V+\dim\g{h}_0^\perp,
\end{align*}
which implies $\dim\g{h}_0^\perp\leq n-\dim V$.
On the other hand, 
\[
\dim\pi(\g{h})
=\dim\xi(\pi(\g{h}))+\dim\ker\xi
\geq \dim\xi(\pi(\g{h}))=\dim\g{w}-\dim V=n-\dim V.
\]
The last two inequalities imply $\dim\g{h}_0^\perp\leq\dim\pi(\g{h})$.
Now recall that $\pi(\g{h})$ is isomorphic to an abelian Lie subalgebra of $\g{so}(\g{h}_0^\perp)$. Then,
\begin{align*}
\dim\g{h}_0^\perp
&{}\leq\dim\pi(\g{h})\leq\mathop{\rm rank}\g{so}(\g{h}_0^\perp)
\leq \frac{1}{2}\dim\g{h}_0^\perp.
\end{align*}
Hence, $\dim\g{h}_0^\perp=0$, which implies $\pi(\g{h})=0$.
Therefore $\g{h}=V=\g{w}\cong \R^n$, and $H$ acts on $\C^n$ by translations by vectors in a Lagrangian subspace $V$ of $\C^n$.
\end{proof}

\subsection{Homogeneous Lagrangian foliations on complex hyperbolic spaces}\label{subsec:chn_proof}\hfill

Let $H$ be a connected Lie subgroup of $G=\s{SU}(1,n)$ with Lie algebra $\g{h}$. Assume that $H$ induces a homogeneous Lagrangian foliation on $\CH^n$, $n\geq 2$. By Proposition~\ref{prop:abelian}, $\g{h}$ is abelian. In particular, it is solvable and, as such, it is contained in a maximal solvable subalgebra of $\g{g}$, which by definition is a Borel subalgebra of $\g{g}$.
Thus, there exists a Borel subalgebra $\g{b}$ containing $\g{h}$.
It turns out that Borel subalgebras of semisimple Lie algebras are well understood in terms of Cartan subalgebras (see~\cite{Mo61} and~\cite{Knapp}).
There is a Cartan decomposition $\g{g}=\g{k}\oplus\g{p}$ such that $\g{b}=\g{t}\oplus\tilde{\g{a}}\oplus\tilde{\g{n}}$, where $\g{t}\subset\g{k}$, $\tilde{\g{a}}\subset\g{p}$, and $\g{t}\oplus\tilde{\g{a}}$ is a Cartan subalgebra of $\g{g}$.
Moreover, $\tilde{\g{g}}_{\tilde{\lambda}}=\{X\in\g{g}:\ad(H)X=\tilde{\lambda}(H)X\text{ for all $H\in\tilde{\g{a}}$}\}$, $\tilde{\lambda}\in\tilde{\g{a}}^*$, is a root space with respect to $\tilde{\g{a}}$, $\tilde{\Sigma}^+$ is the set of positive roots with respect to a certain ordering in $\tilde{\g{a}}$, and thus
$\tilde{\g{n}}=\oplus_{\tilde{\lambda}\in\tilde{\Sigma}^+}\tilde{\g{g}}_{\tilde{\lambda}}$.
There are exactly two conjugacy classes of Cartan subalgebras in $\g{g}$. Indeed, since $\g{a}$ is abelian and $\CH^n$ has rank one, either $\tilde{\g{a}}=0$, or $\tilde{\g{a}}=\g{a}$ is a maximal abelian subspace of $\g{p}$.
If $\tilde{\g{a}}=0$, then $\g{b}=\g{t}$ is a maximal abelian subalgebra of $\g{k}$. In this case $H$ would be contained in $K$. Cartan's fixed point theorem implies that the compact Lie group $K$ fixes a point when acting on the Hadamard manifold $\CH^n$. But since $H\subset K$, $H$ would also have a fixed point, which is not possible because $H$ induces a foliation by assumption. 
Therefore $\tilde{\g{a}}=\g{a}$, which means that $\g{b}=\g{t}\oplus\g{a}\oplus\g{n}$ is a so-called maximally non-compact Borel subalgebra. In this case, $\g{t}\subset\g{k}_0$, and $\g{a}\oplus\g{n}$ coincides, up to conjugation, with the non-compact part of the Iwasawa decomposition described in \S\ref{subsec:chn}. Hence, we have proved

\begin{proposition}\label{prop:Borel}
If $H$ induces a Lagrangian foliation on $\CH^n$, then its Lie algebra $\g{h}$ is, up to conjugation, an abelian subalgebra of a maximally non-compact Borel subalgebra $\g{t}\oplus\g{a}\oplus\g{n}$ with respect to a Cartan decomposition $\g{g}=\g{k}\oplus\g{p}$, where $\g{a}$ is a maximal abelian subspace of $\g{p}$, and $\g{t}$ is a maximal abelian subalgebra of $\g{k}_0$.
\end{proposition}

We now study the projection of $\g{h}\subset \g{t}\oplus\g{a}\oplus\g{n}$ onto $\g{a}\oplus\g{n}$, which we denote by $\g{h}_{\g{a}\oplus\g{n}}$.
Since $H$ induces a Lagrangian foliation, we get that $T_o(H\cdot o)\cong \g{h}_{\g{a}\oplus\g{n}}$ is a Lagrangian subspace of $\g{a}\oplus\g{n}\cong \C^{n}$, where $o\in\CH^n$ is the fixed point of $K$.
As $\g{h}\cap\g{t}$ is the kernel of the orthogonal projection of $\g{h}$ onto $\g{a}\oplus\g{n}$, the first isomorphism theorem yields
\begin{equation}\label{eq:dims}
\dim\g{h}=\dim(\g{h}\cap\g{t})+\dim\g{h}_{\g{a}\oplus\g{n}}=\dim(\g{h}\cap\g{t})+n.
\end{equation}
Now we consider the following linear map
\[
T\colon \g{h}_{\g{a}\oplus\g{n}}\to\g{t}\ominus(\g{h}\cap\g{t}),\quad X\mapsto T_X,
\]
defined by the requirement $T_X+X\in\g{h}$.
Here $\g{t}\ominus(\g{h}\cap\g{t})$ represents the orthogonal complement of $\g{h}\cap\g{t}$ in $\g{t}$ with respect to $\cal{B}_\theta$.
This map is well-defined because $T_X+X$, $T_X'+X\in\g{h}$ implies $T_X-T_X'=(T_X+X)-(T_X'+X)\in\g{h}\cap\g{t}$, and so, $T_X=T_X'$.
It is also easy to see that $T$ is linear.

We now distinguish two cases: the projection of $\g{h}$ onto $\g{a}\oplus\g{g}_{2\alpha}$ is either $\g{a}\oplus\g{g}_{2\alpha}$, or strictly contained in $\g{a}\oplus\g{g}_{2\alpha}$.

\subsection*{Case (i): the orthogonal projection of $\g{h}$ onto $\g{a}\oplus\g{g}_{2\alpha}$ is surjective}
In this case we have $\g{h}_{\g{a}\oplus\g{n}}=\R(B+X) \oplus \g{w} \oplus \R(Y+Z)$, with $\g{w}$ a subspace of $\g{g}_\alpha$, and $X$, $Y\in\g{g}_\alpha\ominus\g{w}$. Our goal will be to show that this case is not possible.

From now on, the symbol $\ominus$ denotes orthogonal complement with respect to the inner product $\langle\cdot,\cdot\rangle$ on $\g{a}\oplus\g{n}$.

\begin{lemma}
	We have $\g{h}\cap\g{t}=0$, $\dim \g{h}=n$, and $X\neq 0\neq Y$ with
	\[
	X=\gamma Y + \frac{1}{\lVert Y\rVert^2}JY.
	\]
\end{lemma}
\begin{proof}
Recall that $\g{h}_{\g{a}\oplus\g{n}}$ is Lagrangian in $\g{a}\oplus\g{n}$.
Hence, $\g{w}$ is totally real in $\g{g}_\alpha\cong\C^{n-1}$, and for each $U\in\g{w}$ we have
$0=\langle J(B+X), U \rangle =\langle JX, U \rangle$ and $0= \langle J(Y+Z), U \rangle =\langle JY, U \rangle$. This implies $\g{w} \subset \g{g}_{\alpha}\ominus (\C X + \C Y)$. Moreover,
\begin{equation}\label{eq:JXY}
0=\langle J(B+X), Y+Z \rangle =1+\langle JX, Y\rangle,
\end{equation}
which in particular yields $X\neq 0\neq Y$. Thus, as $\dim\g{h}_{\g{a}\oplus\g{n}}=n$, we have $\C X= \C Y$, and $\g{w}$ is Lagrangian in $\g{g}_\alpha\ominus\C Y\cong\C^{n-2}$.
From \eqref{eq:JXY}, we can write
$X=\gamma Y + \frac{1}{\lVert Y\rVert^2}JY$,
for~some~$\gamma \in \R$.

Let $U\in\g{w}$.
If $S\in\g{h}\cap\g{t}$ is arbitrary, and taking into account that $\g{h}$ is abelian, we get
$0=[S,T_U+U]=[S,U]$ and $0=[S,T_{Y+Z}+Y+Z]=[S,Y]$.
Using Lemma~\ref{lemma:adJ}, this implies $0=[\g{h}\cap\g{t},\C\g{w}\oplus\C Y]=[\g{h}\cap\g{t},\g{g}_\alpha]$.
Since the connected Lie subgroup $K_0\cong \s{U}(n-1)$ of $G$ with Lie algebra $\g{k}_0\cong\g{u}(n-1)$ acts effectively on $\g{g}_\alpha\cong\C^{n-1}$, we must have $\g{h}\cap\g{t}=0$.
In particular, $\dim\g{h}=n$ by~\eqref{eq:dims}.
\end{proof}

In order to simplify notation we write
\[
T_B=T_{B+\gamma Y+\frac{1}{\lVert Y\rVert^2}JY}\quad \text{ and }\quad
T_Z=T_{Y+Z}.
\]

\begin{lemma}\label{lemma:TZJY}
	We have:
	\begin{equation*}
		[T_Z,JY]=\frac{\lVert Y\rVert^2}{2}Y.
	\end{equation*}
\end{lemma}
\begin{proof}
Let $U\in \g{w}$.
The fact that $\g{h}$ is abelian yields
$0=[T_U+U, T_Z+Y+Z]=[T_U,Y]-[T_Z, U]$, and taking inner product with $Y$ we get $\langle [T_Z,JY],JU\rangle=\langle [T_Z, Y], U \rangle=0$. Using this, together with the facts that $\g{h}$ is abelian and $Y\in\g{g}_\alpha\ominus\C\g{w}$, we obtain
\begin{align*}
0&{}=\langle \Bigl[T_B+B+\gamma Y + \frac{1}{\lVert Y\rVert^2}JY,\, T_U+U\Bigr],\, Y\rangle
=-\langle [T_B, Y], U \rangle-\frac{1}{\lVert Y\rVert^2} \langle [T_U,JY],\, Y \rangle,
\\
0&{}=\langle \Bigl[T_B+B+\gamma Y + \frac{1}{\lVert Y\rVert^2}JY,\, T_Z+Y+Z\Bigr], U \rangle
=\langle [T_B, Y], U \rangle -\frac{1}{\lVert Y\rVert^2} \langle [T_Z, JY], U \rangle,
\\[1ex]
0&{}=\langle [T_U+U,\,T_Z+Y+Z],\, JY\rangle
=-\langle [T_U,JY],Y\rangle+\langle [T_Z,JY],U\rangle,
\end{align*}
which imply $\langle [T_Z, JY], U \rangle=\langle [T_U,JY],Y\rangle=\langle [T_B,Y],U\rangle=0$.
Similarly,
\begin{align*}
0&{}=\langle \Bigl[T_B+B+\gamma Y + \frac{1}{\lVert Y\rVert^2}JY,\, T_Z+Y+Z\Bigr],\, Y \rangle
=\frac{1}{2}\lVert Y\rVert^2-\frac{1}{\lVert Y\rVert^2}\langle [T_Z,JY],Y\rangle
\end{align*}
yields $\langle [T_Z, JY], Y \rangle =\lVert Y\rVert^4/2$.
These calculations imply the result.
\end{proof}

Let $g=\Exp(2JY/\lVert Y\rVert^2)\in AN$. We estimate the dimension of $(\Ad(g)\g{h})_{\g{a}\oplus\g{n}}$. Using the formula for the brackets in $\g{a}\oplus\g{n}$, together with Lemma~\ref{lemma:TZJY}, we get
\begin{align*}
&\Ad(g)(T_Z+Y+Z)
=T_Z+Y+Z-\frac{2}{\lVert Y\rVert^2}[T_Z,JY]-2Z+\frac{2}{\lVert Y\rVert^4}\langle [T_Z,JY],Y\rangle Z=T_Z,
\end{align*}
which has trivial projection onto $\g{a}\oplus\g{n}$.
Since $\dim\Ad(g)\g{h}=\dim\g{h}=n$, we deduce that $\dim(\Ad(g)\g{h})_{\g{a}\oplus\g{n}}\leq n-1$.
But $(\Ad(g)\g{h})_{\g{a}\oplus\g{n}}$ is isomorphic to the tangent space of the orbit of $H$ through $g^{-1}(o)$, so we get an orbit whose dimension is not $n$; in particular, it is not Lagrangian. Case (i) is therefore impossible.

\subsection*{Case (ii): the orthogonal projection of $\g{h}$ onto $\g{a}\oplus\g{g}_{2\alpha}$ is not surjective}
In this case we have $\g{h}_{\g{a}\oplus\g{n}}=\R(aB+X+xZ) \oplus \g{w}$ with $a,x\in\R$, $\g{w}$ a subspace of $\g{g}_\alpha$, and $X\in\g{g}_\alpha\ominus\g{w}$. Our objective will be to show that $\g{h}=\g{w}\oplus\g{g}_{2\alpha}$ for a Lagrangian subspace $\g{w}$ of $\g{g}_\alpha$, from where the Main Theorem will follow. We start with the following lemma.

\begin{lemma}
	We have $\g{h}_{\g{a}\oplus\g{n}}=\g{w}\oplus\g{g}_{2\alpha}$, where $\g{w}$ is a Lagrangian subspace of $\g{g}_\alpha$.
\end{lemma}
\begin{proof}
Since $\g{h}_{\g{a}\oplus\g{n}}$ is a Lagrangian subspace of $\g{a}\oplus\g{n}\cong \C^{n}$, it is clear that $\g{w}$ is totally real in $\g{g}_\alpha\cong\C^{n-1}$.  But then, since $\dim\g{h}_{\g{a}\oplus\g{n}}=n$ by hypothesis, we need $\g{w}$ to be Lagrangian in $\g{g}_\alpha$ and $aB+X+xZ\neq 0$. Moreover, the Lagrangian condition implies that for each $U\in\g{w}$ we have $0=\langle J(aB+X+xZ),U\rangle=-\langle X,JU\rangle$, that is, $X$ is orthogonal to $J\g{w}$.
Since it is also orthogonal to $\g{w}$, and $\g{w}$ is Lagrangian in $\g{g}_\alpha$, we get $X=0$.
Thus, $aB+xZ\neq 0$.  Assume first that $a\neq 0$.
Let $U\in\g{w}$.
Since $\g{h}$ and $\g{t}\oplus\g{a}$ are abelian, $0=2[T_{aB+xZ}+aB+xZ,T_U+U]=2[T_{aB+xZ},U]+aU$.
Equivalently, $(2\ad(T_{aB+xZ})+a\Id)(U)=0$.
Since $\ad(T_{aB+xZ})$ is a skew-symmetric endomorphism of $\g{g}_\alpha$, $-a/2\in\R\setminus\{0\}$ cannot be one of its eigenvalues.
Therefore, $U=0$. Since $n\geq 2$ we would get $\g{w}=0$, which is a contradiction with the fact that $\g{w}$ is Lagrangian in $\g{g}_\alpha$.
All in all this means that $a=0$.
Hence, we get $x\neq 0$ and $xZ\in\g{h}_{\g{a}\oplus\g{n}}$.
Thus, $\g{h}_{\g{a}\oplus\g{n}}=\g{w}\oplus\g{g}_{2\alpha}$.
\end{proof}


\begin{proposition}
	If the orthogonal projection of $\g{h}$ onto $\g{a}\oplus\g{g}_{2\alpha}$ is strictly contained in $\g{a}\oplus\g{g}_{2\alpha}$, then $\g{h}=\g{w}\oplus\g{g}_{2\alpha}$, where $\g{w}$ is a Lagrangian subspace of $\g{g}_\alpha$.
\end{proposition}
\begin{proof}
If $S \in \g{h} \cap \g{t}$ and $U\in\g{w}$, then $0=[S, T_U+U]=[S, U]$, and hence $[\g{h} \cap \g{t},\g{w}]=0$.
Since $\g{w}$ is Lagrangian in $\g{g}_\alpha$, Lemma~\ref{lemma:adJ} ensures that $[\g{h}\cap\g{t},\g{g}_\alpha]=0$.
But the action of $K_0$ on $\g{g}_\alpha$ is effective, so this implies $\g{h}\cap\g{t}=0$, and thus, $\dim\g{h}=n$ by~\eqref{eq:dims}.

We prove that $T_Z=0$.
For each $U\in\g{w}$ we have
$0=[T_U+U, T_Z+Z]=-[T_Z, U]$.
From Lemma~\ref{lemma:adJ} this implies $\ad(T_Z)\g{g}_\alpha=\ad(T_Z)\C\g{w}=0$, and since the action of $K_0$ is effective on $\g{g}_\alpha$, we get $T_Z=0$. In particular $Z\in\g{h}$.

Given $U$, $V\in\g{w}$ we have $0=[T_U+U,T_V+V]=[T_U,V]-[T_V,U]$.
Hence, for $U$, $V$, $W\in\g{w}$ we get
\begin{align*}
\langle[T_U,V],W\rangle
&{}=-\langle V,[T_U,W]\rangle
=-\langle V,[T_W,U]\rangle
=\langle[T_W,V],U\rangle
=\langle[T_V,W],U\rangle\\
&{}=-\langle[T_V,U],W\rangle
=-\langle[T_U,V],W\rangle,
\end{align*}
from where it follows that $\langle[T_U,V],W\rangle=0$, or equivalently,
$[T_U,V]\in\g{g}_\alpha\ominus\g{w}=J\g{w}$.
Using Lemma~\ref{lemma:adJ} this implies $[JV,T_U]=J[V,T_U]\in\g{w}$.
Therefore, for each $W\in\g{w}$ we can define the endomorphism of $\g{w}$
\begin{align}\label{eq:Phi}
\Phi_W\colon \g{w} \to \g{w},\ U \mapsto [JW, T_U].
\end{align}
This map $\Phi_W$ is self-adjoint. Indeed,
\begin{align*}
\langle \Phi_W(U), V \rangle
&{}= \langle [JW, T_U], V \rangle=\langle JW, [T_U, V] \rangle
\\
&{}=\langle JW, [T_V, U] \rangle
=\langle[JW,T_V],U\rangle = \langle \Phi_W(V), U \rangle.
\end{align*}
Therefore, by the spectral theorem, $\Phi_W$ is diagonalizable with real eigenvalues.

We will show that $\Phi_W=0$ for all $W\in\g{w}$.
On the contrary, assume that $\Phi_W \neq 0$ for some non-zero $W\in\g{w}$.
Then there exists an eigenvector $V\in \g{w}$ with a non-zero eigenvalue $\lambda \in \R$,
that is, $[JW,T_V]=\Phi_W(V)=\lambda V$.
Let $g=\Exp(-\frac{1}{\lambda} JW)\in AN$.
Then,
\begin{align*}
&\Ad(g)\Bigl(T_V+V-\frac{1}{2\lambda}\langle W,V\rangle Z\Bigr)
\\
&\qquad{}=T_V+V-\frac{1}{2\lambda}\langle W,V\rangle Z
-\frac{1}{\lambda}[JW,T_V]+\frac{1}{\lambda}\langle W,V\rangle Z
-\frac{1}{2\lambda^2}\langle W,[JW,T_V]\rangle Z
=T_V
\end{align*}
has trivial projection onto $\g{a}\oplus\g{n}$. Since
$n=\dim\g{h}=\dim\Ad(g)\g{h}$ we get $\dim(\Ad(g)\g{h})_{\g{a}\oplus\g{n}}\leq n-1$.
Thus, the orbit of $H$ through $g^{-1}(o)$ would not be a Lagrangian submanifold, which contradicts the assumption that $H$ induces a homogeneous Lagrangian foliation.
Therefore $\Phi_W=0$ for all $W \in \g{w}$, which, by Lemma~\ref{lemma:adJ} and the fact that $\C\g{w}=\g{g}_\alpha$, implies $\ad(T_U)(\g{g}_\alpha)=0$ for all $U\in\g{w}$. Since $K_0$ acts effectively on $\g{g}_\alpha$, we get $T_U=0$ for all $U\in\g{w}$.
Therefore $\g{h}=\g{w} \oplus \g{g}_{2\alpha}$.
\end{proof}

Finally, it follows from the bracket relations of $\g{a}\oplus\g{n}$ that $\g{w}\oplus\g{g}_{2\alpha}$ is an ideal of $\g{a}\oplus\g{n}$.
Hence, for each $g\in AN$, $\Ad(g)\g{h}=\g{h}$, and thus, $gHg^{-1}=H$.
Since $AN$ acts transitively on $\CH^n$ and $H\cdot g^{-1}(o)=g^{-1}(gHg^{-1}\cdot o)=g^{-1}(H\cdot o)$ for all $g\in AN$, this implies that the Lie subgroup $H$ of $G=\s{SU}(1,n)$ whose Lie algebra is $\g{h}=\g{w}\oplus\g{g}_{2\alpha}$, where $\g{w}$ is a Lagrangian subspace of $\g{g}_\alpha$, acts on $\CH^n$ in such a way that all its orbits are congruent to each other. But $T_o(H\cdot o)\cong \g{h}$ is a Lagrangian subspace of $T_o\CH^n\cong\g{a}\oplus\g{n}$, and since $H$ is made of holomorphic isometries, we get that $H\cdot o$, and therefore all other $H$-orbits, are Lagrangian submanifolds. See~\cite{HK} for further details.
This concludes the proof of the Main Theorem.


\begin{thebibliography}{99}
\bibitem{A}
{M.~Audin}, \textit{The topology of torus actions on symplectic manifolds}, Progress in Mathematics, \textbf{93}, Birkh\"{a}user Verlag, 1991.

\bibitem{BG:cag}
L.~Bedulli, A.~Gori, Homogeneous Lagrangian submanifolds, \emph{Comm.\ Anal.\ Geom.}\ \textbf{16} (2008), 591--615.

\bibitem{BCO} {J. Berndt, S. Console, C. Olmos}, \emph{Submanifolds and holonomy,} Second edition, Monographs and Research Notes in Mathematics, CRC Press, Boca Raton, FL, 2016.

\bibitem{DDK:mathz}
{J.~C.~D\'{i}az-Ramos, M.~Dom\'{i}nguez V\'{a}zquez, A.~Kollross}, {Polar actions on complex hyperbolic spaces}, \textit{Math.\ Z.}\ \textbf{287} (2017), no.~3--4, 1183--1213.

\bibitem{Di02}
A.~J.~Di Scala, Minimal homogeneous submanifolds in Euclidean spaces,
\textit{Ann.\ Global Anal.\ Geom.}\ \textbf{21} (2002), 15--18.

\bibitem{Gavela}
D.~\'Alvarez-Gavela, The simplification of singularities of Lagrangian and Legendrian fronts, \emph{Invent.\ Math.}\ \textbf{214} (2018), no.~2, 641--737.

\bibitem{HK}
{T.~Hashinaga, T.~Kajigaya}, {A class of non-compact homogeneous Lagrangian submanifolds in complex hyperbolic spaces}, \textit{Ann.\ Global Anal.\ Geom.}\ \textbf{51} (2017), no.~1, 21--33.

\bibitem{Knapp}
{A. W. Knapp}, \textit{Lie groups beyond an introduction}, Second edition, Progress in Mathematics, \textbf{140}, Birkh\"auser Boston, Inc., Boston, MA, 2002.

\bibitem{MO:mathz} 
H.~Ma, Y.~Ohnita, On Lagrangian submanifolds in complex hyperquadrics and isoparametric hypersurfaces in spheres, \emph{Math.\ Z.}\ \textbf{261} (2009), 749--785.

\bibitem{Mo61}
G.~D.~Mostow, On maximal subgroups of real Lie groups, \textit{Ann.\ of Math.}\
\textbf{74}(2) (1961), 503--517.

\bibitem{PP:tohoku}
D.~Petrecca, F.~Podest\`a, Construction of homogeneous Lagrangian submanifolds in $\CP^n$ and Hamiltonian stability, \emph{Tohoku Math.\ J.\ (2)} \textbf{64} (2012), 261--268.

\bibitem{Weinstein}
A.~Weinstein, Symplectic manifolds and their Lagrangian submanifolds, \emph{Adv.\ Math.}\ \textbf{6} (1971), 329--346. 
\end{thebibliography}
\end{document}